\documentclass{amsart}
\usepackage{amsmath}
\usepackage{amssymb,manfnt}
\usepackage{amsthm,epsfig,rotating}

\date{}

\def\sing{\qopname\relax o{sing}}

\def\rank{\qopname\relax o{rank}}
\def\cube{\text{\mancube}}

\theoremstyle{theorem}
\newtheorem{theorem}{Theorem}
\newtheorem{lemma}{Lemma}
\newtheorem{proposition}{Proposition}
\theoremstyle{remark}
\newtheorem{remark}{Remark}

\theoremstyle{definition}
\newtheorem{definition}{Definition}
\author {Ivan Dynnikov and Alexandra Skripchenko}
\address{\noindent Steklov Mathematical Institute of Russian Academy of Sciences, 8 Gubkina Str., Moscow 119991, Russia}
\email{dynnikov@mech.math.msu.su}
\address{Faculty of Mathematics, National Research University Higher School of Economics, Vavilova St. 7, 112312 Moscow, Russia}
\email{sashaskrip@gmail.com}

\title[Symmetric band complexes of thin type and chaotic sections]{Symmetric band
complexes of thin type and chaotic sections which are not quite chaotic}
\dedicatory{On the occasion of Yu.\,Ilyashenko's 70th birthdate}
\begin{document}
\maketitle

\begin{abstract}
In a recent paper we constructed a family of foliated 2-complexes of thin type whose typical leaves have two
topological ends. Here we present simpler examples of such complexes that are, in addition, symmetric with respect to an
involution and have the smallest possible rank. This allows for constructing a 3-periodic surface in the
three-space with a plane direction such
that the surface has a central symmetry, and the plane sections of the chosen direction are chaotic
and consist of infinitely many connected components. Moreover, typical
connected components of the sections have an asymptotic direction, which is due to the
fact that the corresponding foliation on the surface in the 3-torus is not uniquely ergodic.
\end{abstract}

\section{Introduction}

Our motivation for this work came from the problem about the asymptotic behavior
of plane sections of triply periodic surfaces in $\mathbb R^3$ posed by S.\,P.\,Novikov in \cite{nov82} in
connection with conductivity theory in monocrystals. The physical model where
such sections appeared was studied by I.\,M.\,Lifshitz and his school in 1950--60s.
The surface in the model is the Fermi surface of a normal metal and is defined as
the level surface of the dispersion law in the space of quasimomenta, which topologically is
a $3$-torus. The Fermi surface of a metal can also be considered as a $3$-periodic surface
in the $3$-space.

The model is designed to study the conductivity in a monocrystal at low temperature
under the influence of a constant and uniform magnetic field $H$.
According to the model the trajectories of electron's quasimomentum
are connected components of the sections of the Fermi surface by planes
perpendicular to~$H$.

Novikov suggested to study plane sections of general null-homologous surfaces
in the $3$-torus and asked what asymptotic properties the unbounded
connected components of such sections may have. The problem
can be considered as one about a foliation induced by a closed $1$-form on a closed
oriented surface, but as such it is very specific as there are serious
restrictions on the cohomology class of the $1$-form.

The first result in this area was obtained by A.\,Zorich who discovered
what is now called the integrable case \cite{zor84}. It was shown later
by I.\,Dynnikov that generically either the 
integrable case occurs or there are no open trajectories (trivial case) \cite{dyn-93}.

For non-generic vectors $H$ whose components are dependent over $\mathbb Z$,
S.\,Tsarev constructed examples that do not
fit into the trivial or integrable case, though minimal components of the
induced foliation on the Fermi surface were of genus~$1$, see \cite{dyn97}. 
A situation in which the foliation has a single minimal component of genus~$3$ and
$H$ is completely irrational
was discovered by I.Dynnikov in \cite{dyn97}. Such examples are now referred to as chaotic.

Physical implications from different types of dynamics of the trajectories for
the conductivity tensor are discussed in \cite{m97,MN}.

After the work \cite{dyn08}, where the construction of \cite{dyn97} was reformulated
in different terms, it became clear that the main instrument for studying chaotic
examples coincided with a particular case of an object that was well known in the geometric
group theory and the theory dynamical systems under the name of band complex,
which is a measured foliated 2-complex of certain type. The theory of such
complexes was developed by E.\,Rips, see \cite{bf95}. In a sense,
constructing examples with chaotic dynamics in Novikov's
problem is equivalent to constructing band complexes of thin type consisting
of three bands.

Several years ago A.\,Maltsev drew first author's attention to the fact that a Fermi surface
of any monocrystal is alway centrally symmetric. So, it is natural to single out
the case when our surface has such a symmetry. For the corresponding band complexes
this means that they must be invariant under an involution flipping the transverse orientation of
the foliation. Symmetric band complexes of thin type, which give examples of chaotic dynamic
on a centrally symmetric surface, are constructed by A.\,Skripchenko in \cite{S1}.

The behavior of chaotic trajectories in Novikov's problem is not well understood in general.
One of the interesting questions is how many trajectories may lie in a single plane.
In the theory of band complexes of thin type this is related to the question
about the number of topological ends of a typical leaf.
A~single topological end would imply a single connected component of a typical chaotic section, and two
topological ends would imply infinitely many components (see \cite{dyn08}, \cite{s13}).
This question about possible typical leaf structure of thin type band complexe is also interesting on its own.

Before recently only examples of thin type band complexes had been known in which
almost all leaves had exactly one topological end \cite{bf95,c10,s13}.
In \cite{ds14} we described the reason for that,
which was the self-similarity of the known examples, and constructed examples
of thin type band complexes having two-ended typical leaves. Those examples did not
obey any symmetry, and it was
not clear for a while whether additional symmetry would be an obstruction for
a band complex to have two-ended typical leaves.

Here we show that not only symmetry but also a certain degeneracy is not an obstruction. Quite surprisingly,
the phenomenon can be observed for band complexes that are related
to the so called regular skew polyhedron $\{4,6\,|\,4\}$,
a surface for which the set of all chaotic regimes was explicitly described by I.\,Dynnikov and R.\,de~Leo in \cite{dd}.
Our construction here appears to be even simpler than in~\cite{ds14}.

We also analyze the corresponding chaotic
dynamics on the surface in the $3$-torus. We show that the induced flow, though being minimal,
decomposes into two ergodic components. This appears to be a reason for the existence of an asymptotic direction of
the trajectories in $\mathbb R^3$. In principle, the proofs of these facts are self-contained and
do not use any band complexes.
However, band complexes provide for a more intuitive way to understand the origin of the construction,
and we start the exposition from introducing them.

\section*{Acknowledgements}
The first author is supported in part by Russian Foundation for Rundamental Research (grant no.~13-01-12469).
The second author is partially supported by Lavrentiev Prix and by the Dynasty Foundation.

\section{Band complexes}

We start by recalling basic definitions.

\begin{definition}
\emph{A band} is a (possibly degenerate) rectangular $\mathcal B=[a,b]\times[0,1]\subset\mathbb R^2$, $a\leqslant b$,
endowed with the $1$-form $dx$,
where $x$ is the first coordinate in the plane $\mathbb R^2$.
The horizontal sides $[a,b]\times\{0\}$ and $[a,b]\times\{1\}$ are called \emph{the bases} of the band;
the band is \emph{degenerate} if
$a=b$. The value $(b-a)$ is called \emph{the width} of the band and denoted $|\mathcal B|$.

\begin{definition}
\emph{A band complex} is a $2$-complex $X$ endowed with a closed $1$-form $\omega$ obtained from a union~$D$
of pairwise disjoint closed (possibly degenerate to a point) intervals
of $\mathbb R$, called \emph{the support multi-interval of~$X$},
and several pairwise disjoint bands $\mathcal B_i=[a_i,b_i]\times[0,1]$ by gluing each base
of every band isometrically and preserving the orientation to a closed subinterval of $D$. The form $\omega$ is
the one whose restriction to each band and to $D$ is $dx$, so, we keep using notation $dx$ for it.

The $1$-form $dx$ defines \emph{a singular foliation $\mathcal F_X$} on $X$ whose \emph{leaves} are maximal
path connected subsets of $X$ the restriction of $dx$ to which vanishes.
\emph{Singularities} of $\mathcal F_X$ are such points $p\in X$ that the restriction of
$\mathcal F_X$ to any open neighborhood of $p$ is not a fibration over an open interval. It is easy to see
the set of singular points is the union of vertical sides of all the bands. Leaves
containing a singularity are called \emph{singular}, and otherwise \emph{regular}.

A band complex $Y$ is called \emph{annulus free} if all regular leaves are simply connected.
\end{definition}

\begin{definition}
The dimension
$$\dim_{\mathbb Q}\Bigl\{\int_cdx\;;\;c\in H_1(X,\sing(X);\mathbb Z)\Bigr\},$$
where $\sing(X)$ is the set of all singularities of $\mathcal F_X$, is called
the \emph{rank} of a band complex $X$ and denoted $\rank(X)$.
\end{definition}

\end{definition}

\begin{remark}
Our definition of a band complex is less general than appears in geometric group theory as an instrument for describing actions of free groups on $\mathbb R$-trees (see \cite{bf95} for details). Band complexes also appear as suspension complexes for a generalization of interval exchange transformations (more precisely, it is an analogue of Veech's construction of zippered rectangles, see \cite{V}).
\end{remark}

\begin{definition}
Let $Y_1$ and $Y_2$ be band complexes with support multi-intervals $D_1$ and $D_2$, respectively.
We say that they are \emph{isomorphic}
if there is a homeomorphism $f:Y_1\rightarrow Y_2$ (called then \emph{an isomorphism from $Y_1$ to $Y_2$})
such that we have $f^*(dx)=dx$.
If, additionally, $Y_1$ has minimal possible number of bands among all band complexes
isomorphic to $Y_2$ and we have $f(D_1)\subset D_2$, then the image $f(\mathcal B)$ of any band~$\mathcal B$ of $Y_1$ is called
\emph{a long band of $Y_2$}.
\end{definition}

\begin{definition}
A band complex $X$ is \emph{symmetric} if there exists an involution $\tau:X\rightarrow X$ such that it takes
bands to bands and we have $\tau^*(dx)=-dx$.
\end{definition}

\begin{definition}
\emph{An enhanced band complex} is a band complex $Y$ together with an assignment of a positive
real number to each band. This number is called \emph{the length of the band}.

A band of width $w$ and length $\ell$ is said \emph{to have dimensions $w\times\ell$}.
The product $w\ell$ will be referred to as \emph{the area} of the band.
\emph{The length} of a long band $\mathcal B$ is the sum of the lengths of all bands contained in~$\mathcal B$.

Each band $\mathcal B$ of an enhanced band complex $Y$ will be endowed
with the measure $\mu_Y$ obtained from the standard
Lebesgue measure on $\mathcal B\subset\mathbb R^2$ by a rescaling sa as to
have the total measure of $\mathcal B$ equal to its area.

Two enhanced band complexes $Y_1$ and $Y_2$ are \emph{isomorphic} if there exist an isomorphism $Y_1\rightarrow Y_2$ that preserves the lengths of long bands.
\end{definition}

\begin{definition}
Let $Y$ be an enhanced band complex with support multi-interval $D$.
\emph{A~free arc} of~$Y$ is a maximal open interval $J\subset D$
such that it is covered by one of the bases of bands, and all other bases are disjoint from $J$.

Let $J$ be a free arc and $\mathcal B=[a,b]\times[0,1]$ be the band one of whose bases covers $J$
under the attaching map. Let $(c,d)\subset[a,b]$ be the subinterval such that $(c,d)\times\{0\}$
or $(c,d)\times\{1\}$ is identified with $J$ in~$Y$. Let $Y'$ be the band complex obtained from $Y$ by removing
$J$ from $D$, and $(c,d)\times[0,1]$ from $\mathcal B$ thus
replacing $\mathcal B$ with two smaller bands $\mathcal B'=[a,c]\times[0,1]$ and
$\mathcal B''=[d,b]\times[0,1]$
whose bases are attached to $D$ by the restriction of the attaching maps for the bases of $\mathcal B$.
If this produces an isolated point of $D$ such that only one degenerate band is attached to it
(which may occur if $a=c$ or $b=d$), the point and the band are removed.
We then say that $Y'$ is obtained from $Y$ by \emph{a collapse from a free arc}, see Fig.~\ref{collapse}.

If $Y$ is an enhanced band complex then the lengths of $\mathcal B'$ and $\mathcal B''$ are set to that of $\mathcal B$.
\begin{figure}[ht]
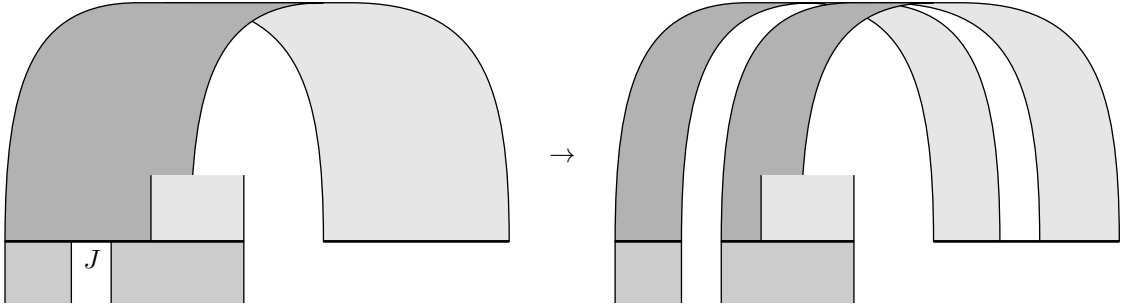

\centerline{\includegraphics{collapse1.eps}\put(-166,15){$J$}\quad
\raisebox{55pt}{$\rightarrow$}\quad
\includegraphics{collapse2.eps}}
\caption{Collapse from a free arc}\label{collapse}
\end{figure}
\end{definition}

\begin{definition}\label{thindef}
An annulus free band complex $Y$ is said to be \emph{of thin type}
if the following two conditions hold:
\begin{enumerate}
\item
every leaf of the foliation $\mathcal F_Y$ is everywhere dense in $Y$;
\item there is an infinite sequence $Y_0=Y,Y_1,Y_2,\ldots$ in which every
$Y_i$, $i\geqslant1$ is a band complex obtained from $Y_{i-1}$ by a collapse from a free arc
(such a sequence is said \emph{to be produced by the Rips machine}).
\end{enumerate}
\end{definition}

\begin{remark}
Again, we use a particular case of a more general notion of a band complex of thin type, which need not necessarily be annulus free.
For a full description of the Rips machine see~\cite{bf95}.
\end{remark}

From the general theory of the Rips machine \cite{bf95} one can extract the following.

\begin{proposition}\label{thin-criteria}
Let $Y$ be a band complex made of three bands. Then the following conditions are equivalent:
\begin{enumerate}
\item $Y$ is of thin type;
\item all leaves of $\mathcal F_Y$ are infinite trees that are not quasi-isometric to a straight line;
\item there are uncountably many leaves of $\mathcal F_Y$ that are
not quasi-isometric to a point, to a straight line, or to a plane.
\end{enumerate}
\end{proposition}

The first example of a  band complex of thin type was constructed by G.Levitt~\cite{l93}.

In \cite{g96} D.Gaboriau asked a question about possible number of topological ends of orbits (or, equivalently, leaves) in the
thin case. It was noted by M.Bestvina and M.Feighn in \cite{bf95} and D.Gaboriau in~\cite{g96}
that all but finitely many leaves of a band complex
of thin type are quasi-isometric to infinite trees with at most two topological
ends, and shown that one-ended and two-ended leaves are always present
and, moreover, there are uncountably many leaves of both kinds.

In \cite{ds14} we constructed the first example when almost all orbits are trees with exactly two topological ends.
However, due to the physical origin of our problem we are also interested to see if such band complexes
exist among symmetric ones.

Below we construct an example with the required symmetry and, in addition,
the highest possible level of degeneracy (just two singular leaves).
The rank of the complex in our example is equal to $3$, the smallest possible
as one can show.

More precisely, we have the following

\begin{theorem}\label{comp}
There exist uncountably many symmetric band complexes $Y$ such that:
\begin{enumerate}
\item $Y$ consists of $3$ bands;
\item $Y$ has rank $3$;
\item $Y$ is of thin type;
\item almost any leave $\mathcal F_Y$ is a 2-ended tree.
\end{enumerate}
\end{theorem}

This theorem will be derived from Proposition~\ref{prop-two-end} below.

We use notation $\overrightarrow\ell$, $\overrightarrow{\ell'}$, $\overrightarrow{\ell_k}$, $\overrightarrow w$,
$\overrightarrow{w'}$ and $\overrightarrow{w_k}$ for
$$\begin{pmatrix}\ell_1&\ell_2&\ell_3&\ell_4\end{pmatrix},\
\begin{pmatrix}\ell_1'&\ell_2'&\ell_3'&\ell_4'\end{pmatrix},\
\begin{pmatrix}\ell_{k1}&\ell_{k2}&\ell_{k3}&\ell_{k4}\end{pmatrix},\
\begin{pmatrix}w_1\\w_2\\w_3\end{pmatrix},\
\begin{pmatrix}w_1'\\w_2'\\w_3'\end{pmatrix},\text{ and }
\begin{pmatrix}w_{1k}\\w_{2k}\\w_{3k}\end{pmatrix},$$
respectively. All the coordinates of these columns and rows will be positive reals.

Let $Z(\overrightarrow w,\overrightarrow\ell)$ be an enhanced band complex shown in Fig.~\ref{z(w,l)}.
\begin{figure}[h]
\centerline{\includegraphics[scale=0.4]{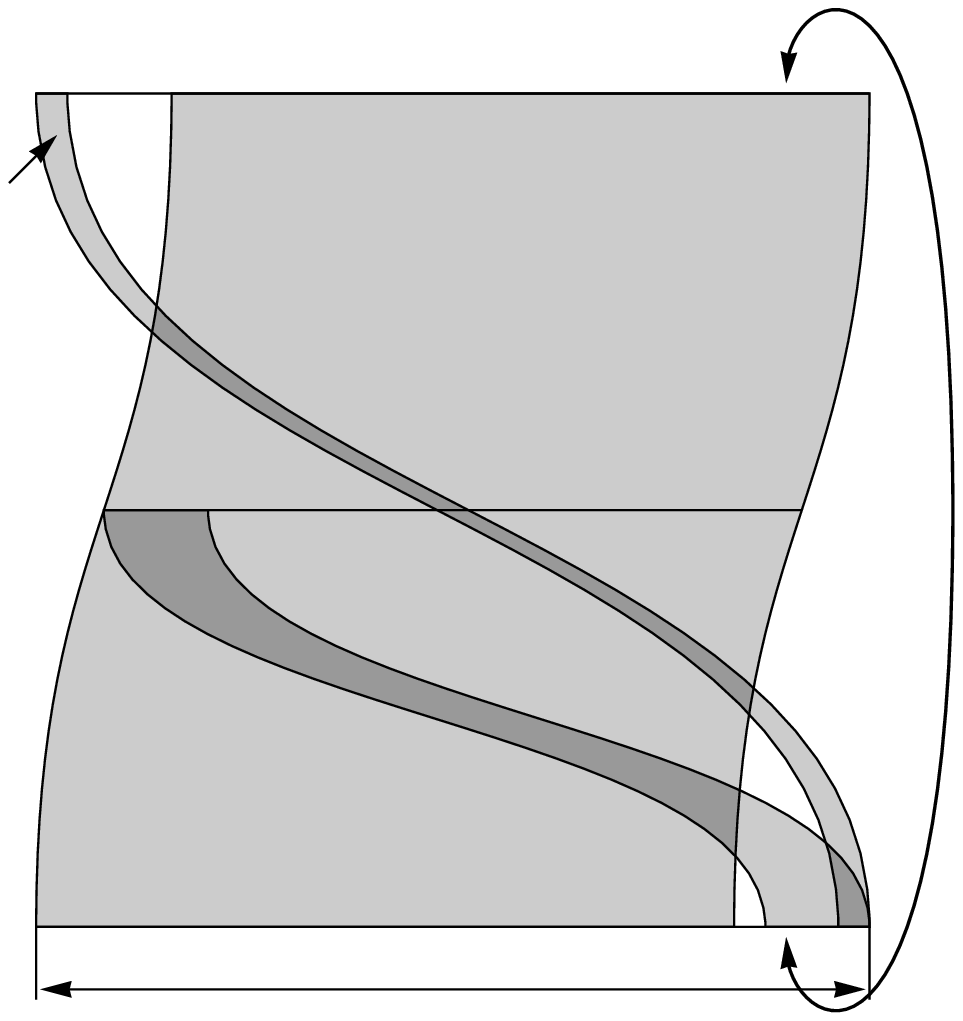}\put(-120,10){\small$w_1+w_2+w_3$}%
\put(-150,40){$\mathcal B_4$}\put(-70,140){$\mathcal B_1$}\put(-40,28){$\mathcal B_2$}\put(-193,155){$\mathcal B_3$}%
\put(-13,80){\begin{sideways}\small identify\end{sideways}}}
\caption{The band complex $Z(\protect\overrightarrow w,\protect\overrightarrow\ell)$}\label{z(w,l)}
\end{figure}
It consists of four
bands $\mathcal B_1$, $\mathcal B_2$, $\mathcal B_3$, and~$\mathcal B_4$ having dimensions $w_1\times\ell_1$,
$w_2\times\ell_2$, $w_3\times\ell_3$,
and $w_4\times\ell_4$, respectively.

Now we define:
\begin{equation}\label{ABmatrices}\begin{aligned}
A(k)&=\begin{pmatrix}
0&0&1&k\\
1&0&0&0\\
0&1&0&0\\
0&1&1&k-1
\end{pmatrix},\qquad
B(k)&=\begin{pmatrix}
k&k&1\\
1&0&0\\
0&1&0
\end{pmatrix}.\end{aligned}\end{equation}
We identify matrices and the linear transformations they define.

Denote: $\mathbb R_+=(0,\infty)$.

\begin{lemma}\label{lem-Rips-step}
Let $\overrightarrow\ell,\overrightarrow{\ell'}\in(\mathbb R_+)^4$, $\overrightarrow w,\overrightarrow{w'}\in(\mathbb R_+)^3$ be related as follows:
$$\overrightarrow\ell A(k)=\overrightarrow{\ell'},\quad
\overrightarrow w=B(k)\overrightarrow{w'},$$
where $k$ is natural number.
Then the enhanced band complex $Z(\overrightarrow{w'},\overrightarrow{\ell'})$ is isomorphic
to one obtained from $Z(\overrightarrow w,\overrightarrow\ell)$ by several collapses from
a free arc.
\end{lemma}
\begin{proof}
It is illustrated in Fig.~\ref{collapses}, where the result of the collapses is shown. One can see that
the obtained band complex is isomorphic to~$Z(\overrightarrow{w'},\overrightarrow{\ell'})$, and $\mathcal B_i'$, $i=1,2,3,4$, are the new bands.
\begin{figure}
\centerline{\includegraphics[scale=0.4]{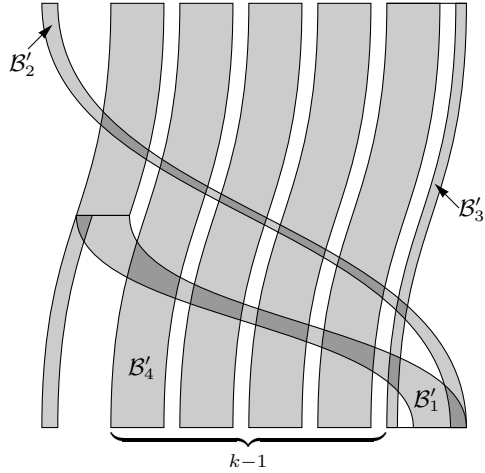}\put(-23,100){\small$\mathcal B_3'$}\put(-40,28){\small$\mathcal B_1'$}%
\put(-193,155){\small$\mathcal B_2'$}\put(-148,40){\small$\mathcal B_4'$}%
\put(-154.5,19){$\underbrace{\hskip3.65cm}_{k-1}$}
}
\caption{Running the Rips machine}\label{collapses}
\end{figure}%
\end{proof}

\begin{lemma}\label{w-exist}
Let $k_0,k_1,k_2,\ldots$ be an arbitrary infinite sequence of natural numbers. Then
there exists an infinite sequence $\overrightarrow{w_0},\overrightarrow{w_1},\overrightarrow{w_2},\ldots$
of points from $(\mathbb R_+)^3$ such that
$$\overrightarrow{w_i}=B(k_i)\overrightarrow{w_i}{}_{+1}.$$
Such a sequence is unique up to scale.
\end{lemma}
\begin{proof}
Let $K=\mathbb R_+^3$ be the positive cone in the three-space
and let $K'=\{\overrightarrow w\in K\;;\;w_3<w_1+w_2\}$. For any
$k,l,m\in\mathbb N$ we have
$$B(k)B(l)B(m)=B'(k,l,m)B'',$$
where
$$B'(k,l,m)=\begin{pmatrix}k(l-1)+1&k(l(m-1)+m)&2k-1\\
l-1&l(m-1)+1&1\\0&m-1&1\end{pmatrix}$$
and
$$B''=\begin{pmatrix}2&1&1\\1&1&0\\1&1&1\end{pmatrix}$$
is a constant matrix.
One can verify that $B''(\overline K)\subset\overline{K'}$, $B''(\overline{K'})\subset K'$, $B'(k,l,m)(K')\subset K'$ for any
$k,l,m\in\mathbb N$. It follows that the linear map $B''$ restricted to $K'$ is a contraction
in the Hilbert projective metric (e.g., see \cite{McM} for the definition and basic properties), and the linear map defined by $B'(k,l,m)$ does not
expand in this metric for any $k,l,m\in\mathbb N$. Therefore, the intersection
$$\bigcap_{i=1}^\infty B(k_1)\ldots B(k_{3i})(K)$$
is a single open ray in $K'$. The claim follows.

\begin{remark}
In \cite{ds14} a flaw occurs in the proof of Lemma~14, where a similar argument is used. The decomposition of
$B(m,n)$ that is given there does not work as proposed. One should use the following decomposition instead:
$$B(m_1,n_1)B(m_2,n_2)=B'(m_1,n_1,m_2,n_2)B'',$$
where
$$B''=\begin{pmatrix}1&1&1&1&1\\1&3&3&1&1\\3&3&1&1&1\\2&4&5&2&2\\1&2&3&2&1\end{pmatrix}.$$
The matrix $B''$ has only positive entries, so it defines a contraction of the positive cone $(\mathbb R_+)^5$ with respect to the
Hilbert projective metric.
It is a direct check that the matrix $B'(m_1,n_1,m_2,n_2)$ has only non-negative entries, so, the
corresponding linear map does not expand the Hilbert metric.
\end{remark}
\end{proof}

Let $\overrightarrow{\ell_0}=(1,1,1,1)$, and let $\overrightarrow{w_0}$ be as in Lemma~\ref{w-exist}.
Define recursively
\begin{equation}\label{l-recursion}
\overrightarrow{\ell_i}{}_{\kern-0.1em+1}=\overrightarrow{\ell_i}\cdot A(k_i).
\end{equation}

\begin{proposition}\label{prop-two-end}
For any sequence $k_0,k_1,k_2,\ldots$ of natural numbers the band complex
$Z(\overrightarrow{w_0},\overrightarrow{\ell_0})$ defined above is annulus free and of thin type.

If, in addition, for all $i\geqslant0$, we have $k_{i+1}\geqslant2k_i$, then
the union of leaves in $Z(\overrightarrow{w_0},\overrightarrow{\ell_0})$ that are not two-ended trees has
zero measure.
\end{proposition}

\begin{proof}
First, we show that $Z(\overrightarrow{w_0},\overrightarrow{\ell_0})$ is annulus free.
One can see from~\eqref{ABmatrices} and \eqref{l-recursion} that all entries of $\overrightarrow{\ell_i}$
grow without bound with $i\rightarrow\infty$. On the other hand, the length of any loop
contained in a leaf of $\mathcal F_{Z(\overrightarrow{w_0},\overrightarrow{\ell_0})}$
is preserved by the Rips machine and should remain fixed. Therefore, all leaves of
$\mathcal F_{Z(\overrightarrow{w_0},\overrightarrow{\ell_0})}$ are simply connected.

Now verify that $Z(\overrightarrow{w_0},\overrightarrow{\ell_0})$ is of thin type. The condition (2)
of Definition~\ref{thindef} is satisfied by Lemma~\ref{lem-Rips-step} and by construction of $\overrightarrow w_0$, so we need only to check
that any leaf of $\mathcal F_{Z(\overrightarrow{w_0},\overrightarrow{\ell_0})}$ is everywhere dense.
By Imanishi's theorem (see \cite{Im} and \cite{GLP94}) the converse would imply the existense of an arc connecting two singularities
of $\mathcal F_Y$ through the regular part of a singular leaf. Such an arc can get only shorter
under a collapse from a free arc, which is inconsistent with the infinite grow of all band lengths.

Now we prove the last claim of the Proposition. Denote for short:
$$A_i=A(k_i),\ B_i=B(k_i),\ Z_i=Z(\overrightarrow{w_i},\overrightarrow{\ell_i}).$$
It follows from Lemma~\ref{lem-Rips-step} that $Z_{i+1}$ can be identified with an enhanced band complex $Z_i$
obtained from~$Z_i$ by a few collapses from a free arc.
So, we think of $Z_{i+1}$ as a subset of $Z_i$ and, hence, of $Z_0$.

Denote by $S_k$ the total area of $Z_i$:
$$S_i=\overrightarrow{\ell_i}\cdot C\cdot \overrightarrow{w_i},$$
where
$$C=\begin{pmatrix}1&0&0\\0&1&0\\0&0&1\\1&0&0\end{pmatrix}.$$

We claim that under the assumptions of the Proposition we have
\begin{equation}\label{limit}
\lim_{i\rightarrow\infty}S_i>0.
\end{equation}

Indeed, it can be checked directly that the matrix
$$\left(A_iC-\Bigl(1-\frac{2}{k_i}\Bigr)CB_i\right)B_{i+1}$$ has only positive entries
for all $i\geqslant0$
since they can be expressed as polynomials in $k_i$ and $(k_{i+1}-2k_i)$ with positive coefficients.
Therefore,
$$S_{i+1}-\Bigl(1-\frac2{k_i}\Bigr)S_i=\overrightarrow{\ell_i}\left(A_iC-\Bigl(1-\frac2{k_i}\Bigr)CB_i\right)
B_{i+1}\overrightarrow w\kern-0.15em{}_{i+2}>0,$$
which can be rewritten as
$$S_{i+1}>\Bigl(1-\frac2{k_i}\Bigr)S_i.$$
Since $k_i$ grows exponentially fast with $i$, we have
$\displaystyle\sum_{i=0}^\infty\frac2{k_i}<\infty$,
which implies~\eqref{limit}.

By definition of a collapse from a free arc the measure $\mu_{Z_{i+1}}$ (see Definition \ref{thindef}) coincides with
the restriction of $\mu_{Z_i}$, and hence of $\mu_{Z_0}$, to $Z_{i+1}$. So, $\displaystyle\lim_{i\rightarrow\infty}S_i$ equals
$\mu_{Z_0}\Bigl(\cap_iZ_i\Bigr)$. By general theory of band complexes (see \cite{bf95}) the subset
$\cap_iZ_i\subset Z_0$ has an empty intersection with one-ended leaves of $\mathcal F_{Z_0}$.
Therefore, the union of two-ended leaves of $\mathcal F_{Z_0}$ has positive measure. Lemma \ref{w-exist}
implies ``a unique ergodicity'' for $\mathcal F_{Z_0}$, namely, that any measurable union of leaves
of $\mathcal F_{Z_0}$ has either zero or full measure. We conclude that
the union of two-ended leaves has full measure.\end{proof}

\begin{proof}[Proof of Theorem \ref{comp}]
Let $Z(\overrightarrow w)$ be a band complex with support interval $D=[0,w_1+w_2+w_3]$ and
three bands $\mathcal B_1$, $\mathcal B_2$, $\mathcal B_3$ whose bases a glued to the following subintervals of $D$:
$$\begin{aligned}
\mathcal B_1&:\text{ to }[0,w_1]\text{ and }[w_2+w_3,w_1+w_2+w_3],\\
\mathcal B_2&:\text{ to }[0,w_1]\text{ and }[w_1+w_3,w_1+w_2+w_3],\\
\mathcal B_3&:\text{ to }[0,w_1]\text{ and }[w_1+w_2,w_1+w_2+w_3].
\end{aligned}$$
So, the band complex $Z(\overrightarrow w)$ can be obtained from the enhanced band
complex $Z(\overrightarrow w,\overrightarrow\ell)$ by collapsing the band $\mathcal B_4$ and forgetting the lengths of the bands.
More precisely, there is a continuous map
$\psi:Z(\overrightarrow w,\overrightarrow\ell)\rightarrow
Z(\overrightarrow w)$ that preserves the $1$-form $dx$ and takes the bands $\mathcal B_1$, $\mathcal B_2$, $\mathcal B_3$
of $Z(\overrightarrow w,\overrightarrow\ell)$ to the respective bands of $Z(\overrightarrow w)$
and takes $\mathcal B_4$ to a subinterval of $D$. Clearly the map $\psi$ takes leaves to leaves and
preserve the quasi-isometry and homotopy class of each leaf. It is also clear that $Z(\overrightarrow w)$ is symmetric
with respect to the involution that flips the support interval $D$.

It follows from Proposition \ref{prop-two-end} that there are uncountably many choices of parameters $\overrightarrow w$
for which almost all leaves of $Z(\overrightarrow w)$ are two-ended trees.
\end{proof}

\section{Plane sections of the regular skew polyhedron $\{4,6\,|\,4\}$}

We recall briefly the formulation of Novikov's problem on plane sections of 3-periodic surfaces.
Let $M$ be a closed null-homologous surface in the 3-torus $\mathbb T^3=\mathbb R^3/L$,
where $L\cong\mathbb Z^3$ is a lattice, and let $H=(H_1,H_2,H_3)\in\mathbb R^3$ be a non-zero
vector. We denote by $p$ the projection $\mathbb R^3\rightarrow\mathbb T^3$, and by $\widehat M\subset\mathbb R^3$
the $\mathbb Z^3$-covering $p^{-1}(M)$ of $M$. We also fix a smooth function $f:\mathbb T^3\rightarrow\mathbb R$
of which $M$ is a level surface, $M=\{x\in\mathbb T^3\;;\;f(x)=c\}$.

Non-singular connected components of the intersection of $\widehat M$ with a plane of the form
\begin{equation}\label{plane}\Pi_a=\{x\in\mathbb R^3\;;\;\langle H,x\rangle=a\},\end{equation}
where $\langle\;,\,\rangle$ stands for the Euclidean scalar product,
are trajectories of the following ODE:
\begin{equation}\label{ode}
\dot x=\nabla\widehat f(x)\times H,\end{equation}
where $\widehat f=f\circ p$.
Their image in $\mathbb T^3$ under $p$ are leaves of the foliation $\mathcal F_M$ on $M$ defined by the kernel of the closed 1-form
\begin{equation}\label{eta}\eta=(H_1\,dx_1+H_2\,dx_2+H_3\,dx_3)|_M.
\end{equation}

Novikov's question was about the existence of an asymptotic direction of open trajectories defined by~\eqref{ode}.
As shown in \cite{dyn-93} the foliation $\mathcal F_M$ typically does not have minimal components of genus larger than one.
For open trajectories this implies that they are typically either not present (in which case we call the pair $(M,H)$ \emph{trivial})
or have \emph{a strong asymptotic direction} (then the pair $(M,H)$ is called \emph{integrable}),
which means that, for a certain parametrization (not related to the one prescribed by~\eqref{ode}), they have the form
\begin{equation}\label{strong}x(s)=sv+O(1),\end{equation}
where $v\in\mathbb R^3$ is a constant vector. There is also a special case discovered by S.\,Tsarev (see \cite{dyn97}) when
minimal components of $\mathcal F_M$ have genus one but the trajectories have an asymptotic direction
only in the usual, not the strong, sense, i.e. with $o(s)$ instead of $O(1)$ in~\eqref{strong}.
In Tsarev's case, the vector $H$ is not ``maximally irrational'',
i.e.\ $\dim_{\mathbb Q}\langle H_1,H_2,H_3\rangle=2$.

It is, however, possible that $\mathcal F_M$ has a minimal component of genus $>1$ (as shown in \cite{dyn97}
the genus cannot be equal to $2$, so `$>1$' actually means `$\geqslant3$' here),
see \cite{dyn97}. In this case, the pair $(M,H)$ is called \emph{chaotic} since
there is a~priori no reason for open trajectories to have an asymptotic direction. 
If the system is chaotic and uniquely ergodic, then, as A.Zorich notes in~\cite{zor92},
trajectories, indeed, cannot have an asymptotic direction.
Particular chaotic examples \cite{dyn97,dyn08,S1} are known in which
almost all planes of the form~\eqref{plane} intersect $\widehat M$ in a single open trajectory, which, in a sense, wanders around the whole plane \cite{s13}.

Chaotic pairs $(M,H)$ can be characterized in terms of any of the foliations $\mathcal F_{N_-}$,
$\mathcal F_{N_+}$ induced by the 1-form
$\omega=H_1\,dx_1+H_2\,dx_2+H_3\,dx_3$ on the submanifolds
$N_-=\{x\in\mathbb T^3\;;\;f(x)\leqslant c\}$, $N_+=\{x\in\mathbb T^3\;;\;f(x)\geqslant c\}$,
of which $M$ if the boundary.
Namely, the following can be extracted from \cite{dyn97}:

\begin{proposition}\label{chaotic-criteria}
A pair $(M,H)$ is chaotic if and only if
$\mathcal F_{N_-}$ (or, equivalently, on $\mathcal F_{N_+}$) has uncountably many leaves that
are not quasi-isometric (in the induced intrinsic
metric) to a point, to a straight line, or to a plane.
\end{proposition}

Since only quasi-isometry class of the leaves matters, one can replace $N_-$ by a foliated
$2$-complex $Z$ embedded in $N_-$ so that every leaf of $Z$ embeds in a leaf of $N_-$
quasi-isometrically. In the genus $3$ case, such a $2$-complex can be chosen among
band complexes made of $3$ bands.

This is how band complexes are related to Novikov's problem in general. Below we demonstrate
this relation explicitly in very detail for a single surface, which was also the main subject
of \cite{dd}, where the set of all $H$'s giving rise to the chaotic case was described.
It appeared to be a fractal set discovered earlier by G.Levitt (\cite{l93}) in connection with pseudogroups of rotations and arose also in
symbolic dynamics (see \cite{AS}). It is shown by A.\,Avila, A.\,Skripchenko, and P.\,Hubert
in \cite{AHS} that the Hausdorff dimension of this set is strictly less than two.

Our $3$-periodic surface $\widehat M$ is going to be the one consisting of all squares of the form
$$\begin{aligned}\{i\}\times[j,j+1]\times[k,k+1],\\
[j,j+1]\times\{i\}\times[k,k+1],\\
 [j,j+1]\times[k,k+1]\times\{i\}\phantom,
 \end{aligned}$$
with $i,\,j,\,k\in\mathbb Z,\ j+k\equiv1(\mathrm{mod}\,2)$.

The fundamental domain of $M$ is shown in Fig.~\ref{surface-fundam}.
\begin{figure}[ht]
\centerline{\includegraphics[scale=0.6]{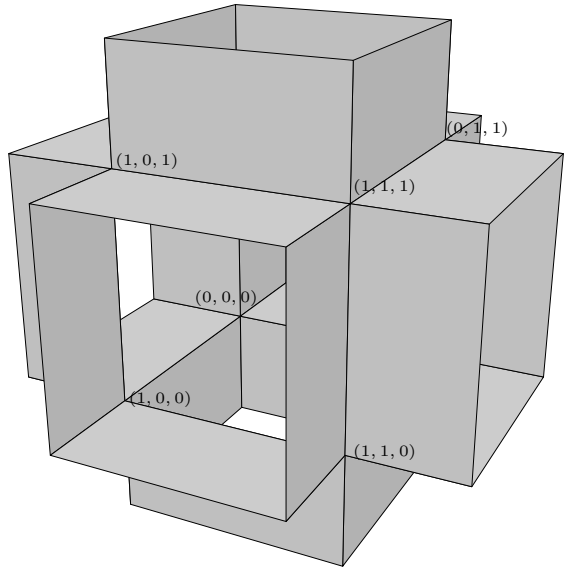}\put(-155,108){\tiny $(0,0,0)$}\put(-180,70){\tiny $(1,0,0)$}%
\put(-95,50){\tiny $(1,1,0)$}\put(-95,150){\tiny $(1,1,1)$}\put(-185,160){\tiny $(1,0,1)$}\put(-60,172){\tiny $(0,1,1)$}}
\caption{A fundamental domain of $\widehat M$}\label{surface-fundam}
\end{figure}
The lattice $L$ is set to $2\mathbb Z^3$. One readily checks that $M=\widehat M/L$
has genus $3$. The surface $\widehat M$ is known in the literature as the regular skew polyhedron $\{4,6\,|\,4\}$,
see~\cite{cox}.

The reader may protest here since the surface $M$ is PL but not smooth. However, for any fixed $H$,
on can smooth it out so as too keep the topology of the foliation $\mathcal F_M$ unchanged. In order
to do so it suffices to $C^0$-approximate $M$ so as to keep the positions of the two monkey saddle singularities
of $\mathcal F_M$ fixed (if $H_1,H_2,H_3>0$, they occur at points $(0,0,0)$  and $(1,1,1)\,(\mathrm{mod}\,L)$)
and to avoid introducing new singularities.

\begin{remark}
Our settings here are in a sense opposite to those of \cite{dyn08}, where the vector $H$
is fixed and the lattice $L$ and the surface $M$ are being varied.
\end{remark}

\begin{proposition}\label{thin-chaotic}
The band complex $Z(\overrightarrow w)$ introduced in the proof of Theorem \ref{comp}
is of thin type if and only if the pair $(M, H)$ is chaotic,
where 
\begin{equation}\label{Hw}2H=(w_2+w_3,w_1+w_3,w_1+w_2).
\end{equation}
\end{proposition}

Note that due to the cubic symmetry of the surface $M$ the pair $(M,H)$ is chaotic if and only if so is $\bigl(M,(|H_1|,|H_2|,|H_3|)\bigr)$.
If all $H_i$'s are positive but don't have the form~\eqref{Hw} with positive $w_i$'s, i.e.\ don't satisfy the triangle
inequalities, then the pair $(M,H)$ is integrable (see \cite{dd}).

\begin{proof}
For $n\in\mathbb Z^3$ we denote:\smallskip\\
by $D(n)$ the straight line segment connecting $n$ with $n+(1,1,1)$;\smallskip\\
by $e_1$, $e_2$, $e_3$ the standard basis of $\mathbb Z^3$;\smallskip\\
by $\mathcal S_i(n)$, $i=1,2,3$ the parallelogram with vertices
$$2n+\Bigl(1-\frac{w_i}{w_1+w_2+w_3}\Bigr)(1,1,1),\quad2n+(1,1,1)\quad2n+2e_i+\frac{w_i}{w_1+w_2+w_3}(1,1,1),\quad2n+2e_i$$
\begin{figure}[ht]
\centerline{\includegraphics[scale=0.5]{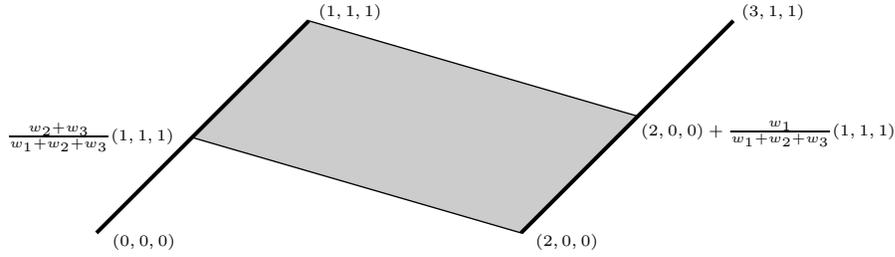}\put(-240,0){\tiny$(0,0,0)$}\put(-162,87){\tiny$(1,1,1)$}%
\put(-80,0){\tiny$(2,0,0)$}\put(-2,87){\tiny$(3,1,1)$}\put(-280,40){\tiny$\frac{w_2+w_3}{w_1+w_2+w_3}(1,1,1)$}%
\put(-40,42){\tiny$(2,0,0)+\frac{w_1}{w_1+w_2+w_3}(1,1,1)$}}
\caption{The strip $S_1(0,0,0)$}\label{astrip}
\end{figure}
see Fig.~\ref{astrip};\smallskip\\
by $\widehat Z$ the union
$$\bigcup_{n\in\mathbb Z^3}\bigl(D(n)\cup\mathcal S_1(n)\cup\mathcal S_2(n)\cup\mathcal S_3(n)\bigr)$$
\begin{figure}[ht]
\centerline{\includegraphics[scale=0.7]{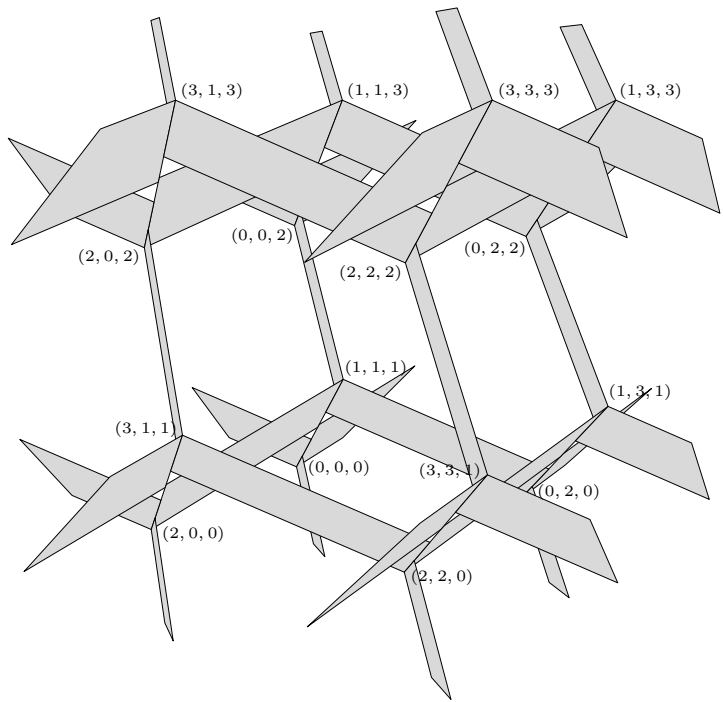}%
\put(-165,85){\tiny$(0,0,0)$}%
\put(-151,123){\tiny$(1,1,1)$}%
\put(-220,60){\tiny$(2,0,0)$}%
\put(-238,100){\tiny$(3,1,1)$}%
\put(-252,165){\tiny$(2,0,2)$}%
\put(-126,44){\tiny$(2,2,0)$}%
\put(-213,228){\tiny$(3,1,3)$}%
\put(-150,228){\tiny$(1,1,3)$}%
\put(-93,228){\tiny$(3,3,3)$}%
\put(-47,228){\tiny$(1,3,3)$}%
\put(-194,173){\tiny$(0,0,2)$}%
\put(-153,159){\tiny$(2,2,2)$}%
\put(-106,168){\tiny$(0,2,2)$}%
\put(-51,114){\tiny$(1,3,1)$}%
\put(-123,84){\tiny$(3,3,1)$}%
\put(-78,76){\tiny$(0,2,0)$}}
\caption{The $2$-complex $\widehat Z$}\label{zhat}
\end{figure}%
\begin{figure}[ht]
\centerline{\includegraphics[scale=0.7]{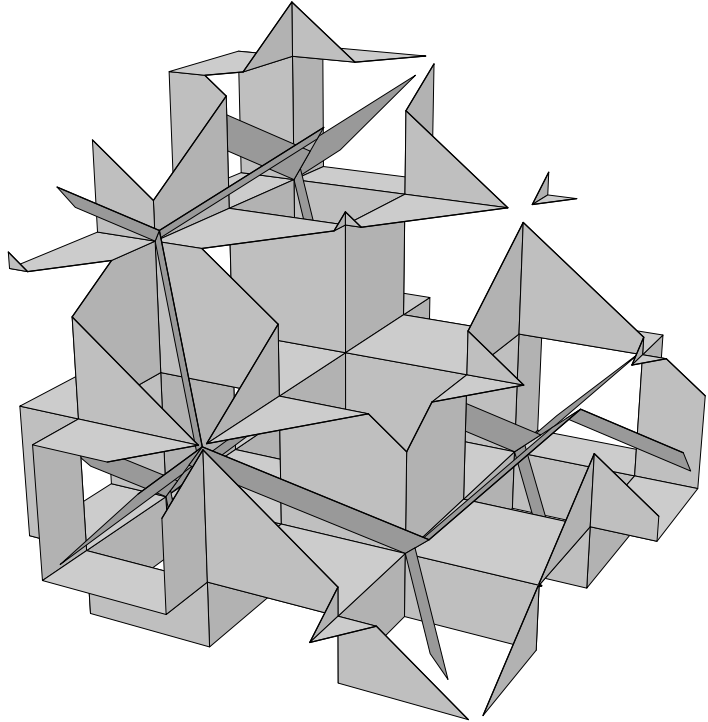}}
\caption{The surface $\widehat M$ and the $2$-complex $\widehat Z$ cut by a plane $\Pi_a$}\label{both}
\end{figure}%
see Figg.~\ref{zhat}, \ref{both};\smallskip\\
by $Z$ the projection $p(\widehat Z)\subset\mathbb T^3$;\smallskip\\
by $\cube_0$ the unit cube $[0,1]\times[0,1]\times[0,1]$;\smallskip\\
by $\cube_0(n)$ the cube $\text{\mancube}_0+2n$;\smallskip\\
by $\cube_i(n)$, $i=1,2,3$, the cube $\text{\mancube}_0+e_i+2n$;\smallskip\\
by $\widehat N_-$ the union
$$\bigcup_{n\in\mathbb Z^3}\bigl(\text{\mancube}_0(n)\cup\text{\mancube}_1(n)\cup\text{\mancube}_2(n)\cup
\text{\mancube}_3(n)\bigr);$$
by $\widehat N_+$ the subset $\widehat N_+$ shifted by the vector $(1,1,1)$;\\
and by $N_-$ (respectively, $N_+$) the projection $p(\widehat N_-)\subset\mathbb T^3$ (respectively,
$p(\widehat N_+)$).

\smallskip

One can readily check the following:\smallskip\\
$N_-\cap N_+=\partial N_-=\partial N_+=M$, $N_-\cup N_+=\mathbb T^3$;\smallskip\\
$D(n)\subset\cube_0(n)$ and
$\mathcal S_i(n)\subset\cube_0(n)\cup\cube_i(n)\cup\cube_0(n+e_i)$ for all $n\in\mathbb Z^3$;\smallskip\\
the two sides of each $\mathcal S_i(n)$, $i=1,2,3$, $n\in\mathbb Z^3$,
that are not parallel to $(1,1,1)$ are orthogonal to $H$;\smallskip\\
the intersection $\Pi_a\cap\cube_0(n)$ is non-empty if and only if so is
$\Pi_a\cap D(n)$;\smallskip\\
the interiors of all the cubes $\cube_i(n)$, $i\in\{0,1,2,3\}$, $n\in\mathbb Z^3$
are pairwise disjoint;\smallskip\\
each $\cube_i(n)$, $i=1,2,3$, $n\in\mathbb Z^3$ shares a face with
$\cube_0(n)$ and with $\cube_0(n+e_i)$, and the rest of the boundary of
$\cube_i(n)$ is disjoint from all other cubes $\cube_j(m)$, $j\in\{0,1,2,3\}$, $m\in\mathbb Z^3$;\smallskip\\
the boundary of the polygon $\Pi_a\cap\cube_i(n)$, $i=1,2,3$, has
non-empty intersection with those of $\Pi_a\cap\cube_0(n)$ and $\Pi_a\cap\cube_0(n+e_i)$ which are not empty;\smallskip\\
the intersection $\Pi_a\cap S_i(n)$, $i=1,2,3$, is a straight line segment
connecting $\Pi_a\cap D(n)$ and $\Pi_a\cap D(n+e_i)$ if
$\Pi_a\cap D(n)\ne\varnothing\ne\Pi_a\cap D(n+e_i)$, and otherwise
empty.

\smallskip

Thus the intersection $\Pi_a\cap\widehat Z$ is a graph $\Gamma_a$ with the set of vertices
$$\Pi_a\cap\Bigl(\bigcup_{n\in\mathbb Z^3}D(n)\Bigr).$$
The intersection $\Pi_a\cap\widehat N_-$ has the following structure. It contains the union
of disjoint discs (some may be degenerate to a point)
$$\Pi_a\cap\Bigl(\bigcup_{n\in\mathbb Z^3}\cube_0(n)\Bigr)$$
in each of which there is a single vertex of $\Gamma_a$. We call these disks \emph{islands}.

The whole intersection $\Pi_a\cap\widehat N_-$ is obtained from the union of islands by
attaching disks of the form $\Pi_a\cap\cube_i(n)$, $i=1,2,3$, $n\in\mathbb Z^3$. Among
such disks there are some whose boundary has a single connected component of intersection with
an island. We call such disks \emph{capes}. An island with all adjoint capes attached is still a disk
containing a single vertex of $\Gamma_a$.

The boundary of any disk of the form $\Pi_a\cap\cube_i(n)$ that is not a cape has
exactly two connected components
in common with islands. We call such a disk \emph{a bridge}.

One can see that two islands are connected by a bridge
if and only if the corresponding vertices of $\Gamma_a$ are connected by an edge, see Fig.~\ref{section}.
\begin{figure}[ht]
\centerline{\includegraphics[scale=0.4]{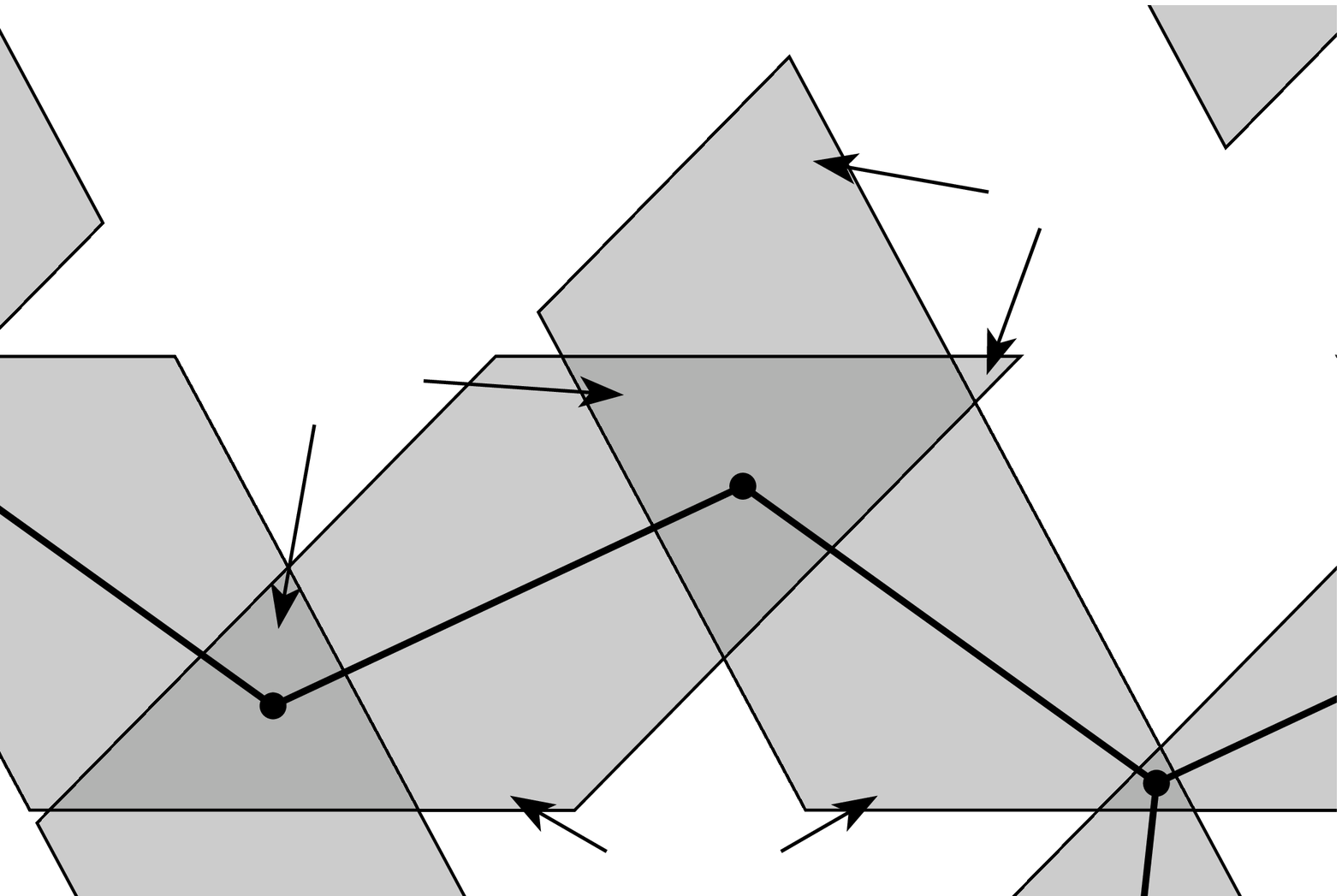}\put(-127,80){$\Pi_a\cap\cube_0(n)$}%
\put(-130,110){$\Pi_a\cap\cube_3(n)$}\put(-172,35){$\Pi_a\cap\cube_1(n)$}\put(-98,22){$\Pi_a\cap\cube_2(n)$}%
\put(-60,125){\small capes}\put(-195,90){\small islands}\put(-130,2){\small bridges}}

\vskip1cm

\centerline{\includegraphics[scale=0.5]{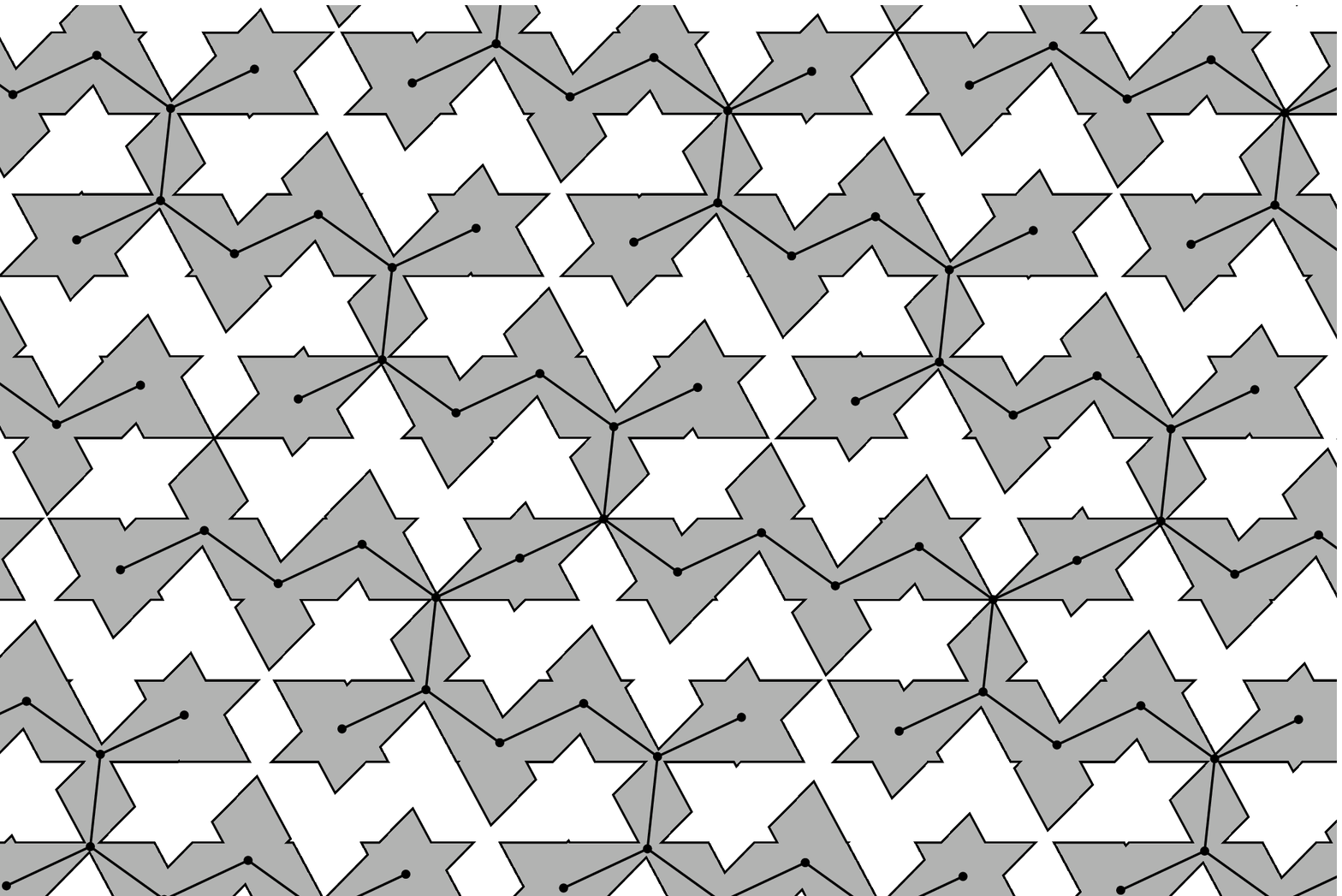}}
\caption{The intersections $\Pi_a\cap\widehat N_-$ and $\Pi_a\cap\widehat Z$}\label{section}
\end{figure}
Since all islands, capes, and bridges have uniformly bounded diameter the inclusion of any
component of $\Gamma_a$ into the corresponding component of $\Pi_a\cap\widehat N_-$ is
a quasi-isometry.

Connected components of $\Pi_a\cap\widehat N_-$ and of $\Gamma_a$ project under $p$ onto
leaves of $\mathcal F_{N_-}$ and $\mathcal F_Z$, respectively. The claim now follows from
Propositions~\ref{thin-criteria} and \ref{chaotic-criteria}.
\end{proof}

\begin{proposition}\label{123}
Let $k_0,k_1,\ldots$ be a sequence of natural numbers such that the series $\displaystyle\sum_{i=0}^\infty\frac1{k_i}$ converges.
Let $\overrightarrow{w_0},\overrightarrow{w_1},\overrightarrow{w_2},\ldots$ be defined as in Lemma~\ref{w-exist} and
$$H=\frac12\begin{pmatrix}0&1&1\\1&0&1\\1&1&0\end{pmatrix}\overrightarrow{w_0}.$$
Then:
\begin{enumerate}
\item the pair $(M,H)$ is chaotic;
\item the foliation $\mathcal F_M$ is not ergodic with respect to the transverse measure $|\eta|$
defined by the $1$-form~\eqref{eta}, there are two ergodic components;
\item almost all connected components of the sections $\Pi_a\cap\widehat M$ have an
asymptotic direction, which is,  up to sign, common for all of them.
\end{enumerate}
\end{proposition}
\begin{proof}
The first claim follows from Propositions~\ref{prop-two-end} and \ref{thin-chaotic}.
It is also a corollary to Lemma~\ref{iit} below.

Let $\xi$ be an oriented simple arc transverse to $\mathcal F_M$ such that $a=\int_\xi\eta>0$.
For any $b\in(0,a]$ we denote by $\xi(b)$ the initial subarc of $\xi$ such that
$\int_{\xi(b)}\eta=b$.

Let $\widehat\xi_1,\widehat\xi_2,\widehat\xi_3$ be the transversals of $\mathcal F_{\widehat M}$,
starting at $(0,0,0)$ composed of the following straight line segments
$$\begin{aligned}
\widehat\xi_1&:[(0,0,0),(0,0,1)]\cup[(0,0,1),(0,1,1)],\\
\widehat\xi_2&:[(0,0,0),(1,0,0)]\cup[(1,0,0),(1,0,1)]\cup[(1,0,1),(1,1,1)],\\
\widehat\xi_3&:[(0,0,0),(0,1,0)]\cup[(0,1,0),(1,1,0)]\cup[(1,1,0),(1,1,1)],
\end{aligned}$$
and let $\xi_i=p(\widehat{\xi_i})$, $i=1,2,3$. We have
$$\int_{\xi_1}\eta=w_{10}+\frac{w_{20}+w_{30}}2,\quad\int_{\xi_2}\eta=\int_{\xi_3}\eta=w_{10}+w_{20}+w_{30},$$
so we have $\int_{\xi_i}\eta>w_{10}$ for all $i=1,2,3$.

Let $R(k)$ be the matrix of the following linear transformation of $\mathbb R^9$:
$$\begin{pmatrix}x_1\\x_2\\x_3\\x_4\\x_5\\x_6\\x_7\\x_8\\x_9\end{pmatrix}\mapsto
\begin{pmatrix}(k-1)(x_2+x_3+x_7+x_8+x_9)+x_7\\
x_8\\x_2+x_9\\x_3+x_7\\x_1+x_2+x_3+x_4\\
(k-1)(x_2+x_3+x_7+x_8+x_9)+x_2+x_3+x_7\\
x_5\\x_2+x_3+x_6\\
(k-1)(x_1+x_2+x_3+x_4+x_8)+x_4
\end{pmatrix}.$$

\begin{lemma}\label{iit}
The first return map defined on 
\begin{equation}\label{transversals}
\xi_1(w_{1i})\cup\xi_2(w_{1i})\cup\xi_3(w_{1i})
\end{equation}
by the foliation $\mathcal F_M$
(for a proper orientation of leaves) endowed
with the invariant measure $|\eta|$ is an interval exchange map with permutation
\begin{equation}\label{permutation}
\Bigl(\,\begin{matrix}1\ 2\ 3\ 4\\3\ 7\ 6\end{matrix}\,\Bigr|\,
\begin{matrix}5\ 6\\4\ 8\ 1\end{matrix}\,\Bigr|\,
\begin{matrix}7\ 8\ 9\\9\ 2\ 5\end{matrix}\,\Bigr)
\end{equation}
and vector of parameters
\begin{equation}\label{w->x}
\overrightarrow{x_i}=(w_{1i}-w_{2i}-w_{3i},\ w_{3i},\ w_{2i}-w_{3i},\ w_{3i},\ w_{2i},\ w_{1i}-w_{2i},
\ w_{3i},\ w_{2i},\ w_{1i}-w_{2i}-w_{3i})^\intercal.
\end{equation}

Let $\mu$ be another invariant transverse measure for $\mathcal F_M$, and let $\overrightarrow{y_i}$ be
the vector of parameters of the corresponding interval exchange map induced on
the union of transversals~\eqref{transversals} (with the same numbering as for $\overrightarrow{x_i}$).
Then for all $i=0,1,2,3,\ldots$ we have
\begin{equation}\label{invmeasures}
\overrightarrow{y_i}{}_{\kern-0.1em+1}=R(k_i)\overrightarrow{y_i},\quad\overrightarrow{y_i}\in V.
\end{equation}
Any sequence $\overrightarrow{y_0},\overrightarrow{y_1},\ldots\in\mathbb R^9_+$ satisfying~\eqref{invmeasures}
defines an invariant transverse measure for~$\mathcal F_M$.
\end{lemma}

We refer the reader to \cite{Vi06} for a detailed account on interval exchange transformations and on the Rauzy--Veech
induction. Here we use a slightly modified version of the standard construction by taking a union of
three transverse arcs instead of just one. That's why we subdivided each row in~\eqref{permutation} into
three blocks that correspond to $\xi_1(w_{1k})$, $\xi_2(w_{1k})$, and $\xi_3(w_{1k})$ (not in this order if $k\ne0\,(\mathrm{mod}3)$).

\begin{proof}
Note that by definition of $\overrightarrow{w_k}$ we always have $w_{1k}>w_{2k}+w_{3k}$ and $w_{2k}>w_{3k}$.

For $k=0$ the claim of the Lemma is obtained by a direct routine check. 
The surface $M$ is cut into $9$ strips each foliated by arcs.
Preimages of the strips in $\mathbb R^3$ are shown in Fig.~\ref{strips}.
\begin{figure}[ht]
\centerline{\includegraphics{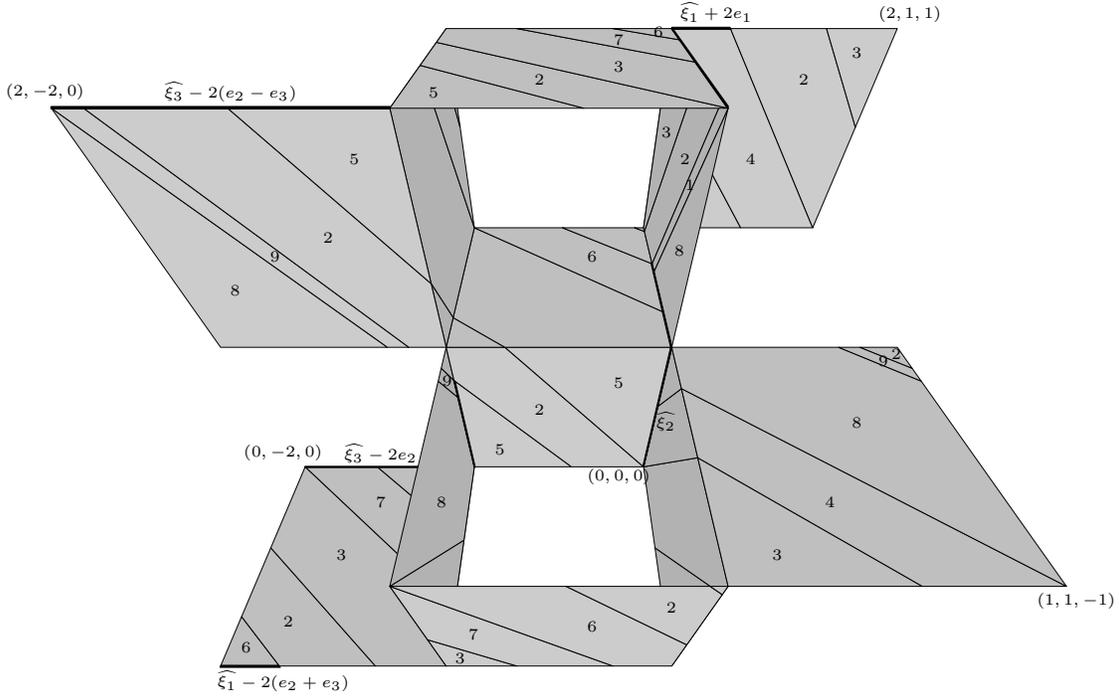}\put(-190,80){\tiny$(0,0,0)$}\put(-320,89){\tiny$(0,-2,0)$}\put(-80,255){\tiny$(2,1,1)$}%
\put(-410,226){\tiny$(2,-2,0)$}\put(-20,33){\tiny$(1,1,-1)$}%
\put(-153,190){\tiny$1$}\put(-155,200){\tiny$2$}\put(-160,30){\tiny$2$}\put(-305,25){\tiny$2$}\put(-110,230){\tiny$2$}\put(-210,230){\tiny$2$}%
\put(-210,105){\tiny$2$}\put(-290,170){\tiny$2$}\put(-162,210){\tiny$3$}\put(-120,50){\tiny$3$}\put(-285,50){\tiny$3$}\put(-90,240){\tiny$3$}%
\put(-240,11){\tiny$3$}\put(-180,235){\tiny$3$}\put(-130,200){\tiny$4$}\put(-100,70){\tiny$4$}\put(-180,115){\tiny$5$}\put(-280,200){\tiny$5$}%
\put(-250,225){\tiny$5$}\put(-225,90){\tiny$5$}\put(-190,163){\tiny$6$}\put(-190,23){\tiny$6$}\put(-165,248){\tiny$6$}\put(-321,15){\tiny$6$}%
\put(-270,70){\tiny$7$}\put(-235,20){\tiny$7$}\put(-180,245){\tiny$7$}\put(-247,70){\tiny$8$}\put(-90,100){\tiny$8$}\put(-325,150){\tiny$8$}%
\put(-157,165){\tiny$8$}\put(-75,126){\tiny$2$}\put(-245,116){\tiny$9$}\put(-80,123.5){\tiny$9$}\put(-310,163){\tiny$9$}%
\put(-164,100){\tiny$\widehat{\xi_2}$}\put(-282,88){\tiny$\widehat{\xi_3}-2e_2$}\put(-155,255){\tiny$\widehat{\xi_1}+2e_1$}%
\put(-330,2){\tiny$\widehat{\xi_1}-2(e_2+e_3)$}\put(-350,225){\tiny$\widehat{\xi_3}-2(e_2-e_3)$}}
\caption{Cutting $M$ into $9$ strips}\label{strips}
\end{figure}

Let $\overrightarrow x\in V$ and $\overrightarrow y=R(k)\overrightarrow x$, $k\in\mathbb N$. We claim that
we can run the Rauzy--Veech induction starting from the permutation~\eqref{permutation} and the
vector of parameters $\overrightarrow y$ so that to obtain after several steps an interval exchange map
with the same permutation and the vector of parameters $\overrightarrow x$. Relations~\eqref{match}
guarantee that, for any $i=1,2,3$, the sum of parameters corresponding to the $i$th block is
the same for the top and bottom rows of the permutation.

The procedure will be slightly more general than usually since we are using three transversals instead of
one. This simply means that we can exchange the blocks synchronously in both rows, so, the process
is not uniquely defined by the initial data. Fig.~\ref{rauzy} shows how the Rauzy--Veech induction
can be run. The transition between any two subsequent lines is the result of several steps
of the ordinary Rauzy--Veech induction with the same winner or just a permutation of blocks.
Each line displays the current permutation, the vector of parameters, and relations (if any) used to
obtain the subsequent transition.

\begin{figure}[ht]
$$\begin{aligned}
&\Bigl(\,\begin{matrix}1\ 2\ 3\ 4\\3\ 7\ 6\end{matrix}\,\Bigr|\,
\begin{matrix}5\ 6\\4\ 8\ 1\end{matrix}\,\Bigr|\,
\begin{matrix}7\ 8\ 9\\9\ 2\ 5\end{matrix}\,\Bigr)
\ &&
(y_1,y_2,y_3,y_4,y_5,y_6,y_7,y_8,y_9)\ &&
y_9=x_4+(k-1)(y_2+y_5)\\
&\Bigl(\,\begin{matrix}1\ 2\ 3\ 4\\3\ 7\ 6\end{matrix}\,\Bigr|\,
\begin{matrix}5\ 6\\4\ 8\ 1\end{matrix}\,\Bigr|\,
\begin{matrix}7\ 8\ 9\\9\ 2\ 5\end{matrix}\,\Bigr)
\ &&
(y_1,y_2,y_3,y_4,y_5,y_6,y_7,y_8,x_4)\ &&
y_5=x_1+x_2+x_3+x_4\\
&\Bigl(\,\begin{matrix}1\ 2\ 3\ 4\\3\ 7\ 6\end{matrix}\,\Bigr|\,
\begin{matrix}5\ 9\ 6\\4\ 8\ 1\end{matrix}\,\Bigr|\,
\begin{matrix}7\ 8\\9\ 2\ 5\end{matrix}\,\Bigr)
\ &&
(y_1,y_2,y_3,y_4,x_1+x_2+x_3,y_6,y_7,y_8,x_4)\ &&
\\
&\Bigl(\,\begin{matrix}7\ 8\\9\ 2\ 5\end{matrix}\,\Bigr|\,
\begin{matrix}1\ 2\ 3\ 4\\3\ 7\ 6\end{matrix}\,\Bigr|\,
\begin{matrix}5\ 9\ 6\\4\ 8\ 1\end{matrix}
\,\Bigr)
\ &&
(y_1,y_2,y_3,y_4,x_1+x_2+x_3,y_6,y_7,y_8,x_4)\ &&
y_6=y_1+x_2+x_3\\
&\Bigl(\,\begin{matrix}7\ 8\\9\ 2\ 5\end{matrix}\,\Bigr|\,
\begin{matrix}1\ 2\ 3\ 4\\3\ 7\ 6\ 1\end{matrix}\,\Bigr|\,
\begin{matrix}5\ 9\ 6\\4\ 8\end{matrix}
\,\Bigr)
\ &&
(y_1,y_2,y_3,y_4,x_1+x_2+x_3,x_2+x_3,y_7,y_8,x_4)\ &&y_8=x_2+x_3+x_6
\\
&\Bigl(\,\begin{matrix}7\ 8\ 6\\9\ 2\ 5\end{matrix}\,\Bigr|\,
\begin{matrix}1\ 2\ 3\ 4\\3\ 7\ 6\ 1\end{matrix}\,\Bigr|\,
\begin{matrix}5\ 9\\4\ 8\end{matrix}
\,\Bigr)
\ &&
(y_1,y_2,y_3,y_4,x_1+x_2+x_3,x_2+x_3,y_7,x_6,x_4)\ &&\\
&\Bigl(\,\begin{matrix}5\ 9\\4\ 8\end{matrix}\,\Bigr|\,
\begin{matrix}7\ 8\ 6\\9\ 2\ 5\end{matrix}\,\Bigr|\,
\begin{matrix}1\ 2\ 3\ 4\\3\ 7\ 6\ 1\end{matrix}
\,\Bigr)
\ &&
(y_1,y_2,y_3,y_4,x_1+x_2+x_3,x_2+x_3,y_7,x_6,x_4)\ &&y_1=(k-1)(y_2+y_3+y_4)+x_7\\
&\Bigl(\,\begin{matrix}5\ 9\\4\ 8\end{matrix}\,\Bigr|\,
\begin{matrix}7\ 8\ 6\\9\ 2\ 5\end{matrix}\,\Bigr|\,
\begin{matrix}1\ 2\ 3\ 4\\3\ 7\ 6\ 1\end{matrix}
\,\Bigr)
\ &&
(x_7,y_2,y_3,y_4,x_1+x_2+x_3,x_2+x_3,y_7,x_6,x_4)\ &&y_4=x_3+x_7\\
&\Bigl(\,\begin{matrix}5\ 9\\4\ 1\ 8\end{matrix}\,\Bigr|\,
\begin{matrix}7\ 8\ 6\\9\ 2\ 5\end{matrix}\,\Bigr|\,
\begin{matrix}1\ 2\ 3\ 4\\3\ 7\ 6\end{matrix}
\,\Bigr)
\ &&
(x_7,y_2,y_3,x_3,x_1+x_2+x_3,x_2+x_3,y_7,x_6,x_4)\ &&y_2=x_8,\ y_7=x_5
\\
&\Bigl(\,\begin{matrix}5\ 9\\4\ 1\ 8\end{matrix}\,\Bigr|\,
\begin{matrix}7\ 8\ 6\ 4\\9\ 2\ 5\end{matrix}\,\Bigr|\,
\begin{matrix}1\ 2\ 3\\3\ 7\ 6\end{matrix}
\,\Bigr)
\ &&
(x_7,x_8,y_3,x_3,x_1+x_2+x_3,x_2,x_5,x_6,x_4)\ &&y_3=x_2+x_9
\\
&\Bigl(\,\begin{matrix}5\ 9\\4\ 1\ 8\end{matrix}\,\Bigr|\,
\begin{matrix}7\ 8\ 6\ 4\\9\ 2\ 5\end{matrix}\,\Bigr|\,
\begin{matrix}1\ 2\ 3\\3\ 6\ 7\end{matrix}
\,\Bigr)
\ &&
(x_7,x_8,x_9,x_3,x_1+x_2+x_3,x_2,x_5,x_6,x_4)\ &&
\\
&\Bigl(\,\begin{matrix}5\ 9\\4\ 1\ 8\end{matrix}\,\Bigr|\,
\begin{matrix}7\ 8\ 6\ 4\\9\ 2\ 5\end{matrix}\,\Bigr|\,
\begin{matrix}1\ 2\ 3\\3\ 6\ 7\end{matrix}
\,\Bigr)
\ &&
(x_7,x_8,x_9,x_3,x_1+x_2+x_3,x_2,x_5,x_6,x_4)\ &&
\\
&\Bigl(\,\begin{matrix}1\ 2\ 3\\3\ 6\ 7\end{matrix}\,\Bigr|\,
\begin{matrix}5\ 9\\4\ 1\ 8\end{matrix}\,\Bigr|\,
\begin{matrix}7\ 8\ 6\ 4\\9\ 2\ 5\end{matrix}
\,\Bigr)
\ &&
(x_7,x_8,x_9,x_3,x_1+x_2+x_3,x_2,x_5,x_6,x_4)\ &&
\\
&\Bigl(\,\begin{matrix}1\ 2\ 3\\3\ 6\ 7\end{matrix}\,\Bigr|\,
\begin{matrix}5\ 6\ 4\ 9\\4\ 1\ 8\end{matrix}\,\Bigr|\,
\begin{matrix}7\ 8\\9\ 2\ 5\end{matrix}
\,\Bigr)
\ &&
(x_7,x_8,x_9,x_3,x_1,x_2,x_5,x_6,x_4)\ &&
\\
&\Bigl(\,
\begin{matrix}5\ 6\ 4\ 9\\4\ 1\ 8\end{matrix}\,\Bigr|\,
\begin{matrix}7\ 8\\9\ 2\ 5\end{matrix}\,\Bigr|\,
\begin{matrix}1\ 2\ 3\\3\ 6\ 7\end{matrix}\,\Bigr)
\ &&
(x_7,x_8,x_9,x_3,x_1,x_2,x_5,x_6,x_4)\ &&
\\\end{aligned}$$
\caption{Running the Rauzy--Veech induction}\label{rauzy}
\end{figure}

After reordering parameters in the last line we get the original permutation with the same subdivision into blocks.

It remains to check that vectors $\overrightarrow{x_i}$ defined by~\eqref{w->x} (written as columns) satisfy
\begin{equation}\label{x=Rx}
\overrightarrow{x_i}=R(k_i)\overrightarrow{x_i}{}_{\kern-0.1em+1}
\end{equation}
for all $i\geqslant0$,
which is straightforward.
\end{proof}

Let $V$ be the subset of $\mathbb R_+^9$ defined by the equations
\begin{equation}\label{match}
x_1+x_4-x_6=x_5-x_8=x_7-x_2.
\end{equation}
It is invariant under $R(k)$ for any $k>0$.

For $i\geqslant0$ denote
$$V_i=R(k_0)R(k_1)\ldots R(k_i)(V),\qquad V_\infty=\bigcap_{i=0}^\infty V_i.$$
We obviously have $V\supset V_0\supset V_1\supset\ldots$.

\begin{lemma}\label{uv}
The subset $V_\infty$ has the form
\begin{equation}\label{vinfty}
V_\infty=\{\alpha\overrightarrow u+\beta\overrightarrow v\;;\;\alpha,\beta\geqslant0,\ (\alpha,\beta)\ne(0,0)\}
\end{equation}
for some non-collinear $\overrightarrow u,\overrightarrow v\in V$. They can be chosen so as to have
$\overrightarrow{x_0}=\overrightarrow u+\overrightarrow v$.
\end{lemma}

\begin{proof}
The matrix $R(k)$ can be written in the form $R(k)=kR'+R''$ with $R'$, $R''$ not depending on $k$, in a unique way.
We will have
$$(R')^4=(R')^2.$$
Since the series $\displaystyle\sum_{i=0}^\infty\frac1{k_i}$ converges this implies that the limit
$$\widetilde R=\lim_{i\rightarrow\infty}\frac{R(k_0)}{k_0}\frac{R(k_1)}{k_1}\ldots\frac{R(k_{2i-1})}{k_{2i-1}}=
\lim_{i\rightarrow\infty}\Bigl(R'+\frac1{k_0}R''\Bigr)\Bigl(R'+\frac1{k_1}R''\Bigr)\ldots \Bigl(R'+\frac1{k_{2i-1}}R''\Bigr)$$
exists and satisfies the relation 
\begin{equation}\label{rtilde}\widetilde R\cdot(R')^2=\widetilde R.\end{equation}

The product of six matrices of the form $R(k)$ with $k>0$ has only positive entries. This implies $(V_i\cup\{0\})\supset\overline V_{\kern-0.2emi+6}$ and
$$V_\infty\cup\{0\}=\bigcap_{i=0}^\infty\overline V_{\kern-0.2emi}=\overline V_{\kern-0.2em\infty}.$$
Together with~\eqref{rtilde} this gives
$$V_\infty\cup\{0\}=\widetilde R(\overline V)=\widetilde RR'(\overline V).$$
Denote: $\overrightarrow{u_\infty}=(1,0,0,0,0,1,0,0,0)^\intercal$, $\overrightarrow{v_\infty}=(0,0,0,0,0,0,0,0,1)^\intercal$.
One easily checks the following
$$R'(\overline V)=\{\alpha\overrightarrow{u_\infty}+\beta\overrightarrow{v_\infty}\;;\;\alpha,\beta\geqslant0\},
\quad R'\overrightarrow{u_\infty}=\overrightarrow{v_\infty},\quad R'\overrightarrow{v_\infty}=\overrightarrow{u_\infty}.$$
Thus, \eqref{vinfty} holds for
$$\overrightarrow u=\widetilde R\overrightarrow{u_\infty},\text{\quad and\quad}\overrightarrow v=\widetilde R\overrightarrow{v_\infty}.$$

The matrix $R(k)$ is invertible for all $k$, so we can set $\overrightarrow{u_0}=\overrightarrow u$, $\overrightarrow{v_0}=\overrightarrow v$
\begin{equation}\label{uvidef}
\overrightarrow{u_i}{}_{\kern-0.1em+1}=k_iR(k_i)^{-1}\overrightarrow{u_i},\quad
\overrightarrow{v_i}{}_{\kern-0.1em+1}=k_iR(k_i)^{-1}\overrightarrow{v_i}\quad\text{for }i\geqslant0.
\end{equation}
We will have
$$\lim_{i\rightarrow\infty}\overrightarrow{u_{2i}}=\lim_{i\rightarrow\infty}\overrightarrow{v_2}_{\kern-0.1em i+1}=
\overrightarrow{u_\infty},\quad
\lim_{i\rightarrow\infty}\overrightarrow{v_{2i}}=\lim_{i\rightarrow\infty}\overrightarrow{u_2}_{\kern-0.1em i+1}=
\overrightarrow{v_\infty},$$
which implies that $\overrightarrow u$ and $\overrightarrow v$ are not collinear.

Now $\overrightarrow{x_0}\in V_\infty$, so, $\overrightarrow{x_0}$ is a non-trivial linear combination $\alpha\overrightarrow u+
\beta\overrightarrow v$. From~\eqref{x=Rx} and \eqref{uvidef} we have
$$\overrightarrow{x_2}_{\kern-0.1em i}=\frac1{\prod_{j=0}^{2i-1}k_i}(\alpha\overrightarrow{u_2}_{\kern-0.1em i}+
\beta\overrightarrow{v_2}_{\kern-0.1em i}).$$
From the definition of $\overrightarrow{w_i}$ (see Lemma~\ref{w-exist}) it follows that
$$\lim_{i\rightarrow\infty}\frac{w_{2i}}{w_{1i}}=\lim_{i\rightarrow\infty}\frac{w_{3i}}{w_{1i}}=0$$
if $k_i\rightarrow\infty$. Together with~\eqref{w->x} this implies
$$\overrightarrow{x_i}=w_{1i}\bigl((1,0,0,0,0,1,0,0,1)^\intercal+o(1)\bigr)=
w_{1i}\bigl(\overrightarrow{u_\infty}+\overrightarrow{v_\infty}+o(1)\bigr)\quad(i\rightarrow\infty),$$
hence $\alpha=\beta$, and by rescaling $\overrightarrow{x_0}$ we can make $\alpha=\beta=1$.
\end{proof}

We can now finilize the proof of Proposition~\ref{123}. It follows from Lemmas~\ref{iit} and \ref{uv} that
$\mathcal F_M$ admits two invariant ergodic transverse measures, $\mu$ and $\nu$, say,
that correspond to the vectors $\overrightarrow u$ and $\overrightarrow v$ from Lemma~\ref{uv},
and, for an appropriate normalization, we have $|\eta|=\mu+\nu$. Let $c_\mu,c_\nu\in H_1(M;\mathbb R)$
be the asymptotic cycles of $\mu$ and $\nu$, respectively (see~\cite{zor99} for the definition),
and $\mathcal E_\mu$, $\mathcal E_\nu$ the respective ergodic components of $\mathcal F_M$.

Denote by $\iota$ the inclusion $M\hookrightarrow\mathbb T^3$. Since $\eta$ is
a restriction of a closed $1$-form in $\mathbb T^3$, we have 
\begin{equation}\label{mu+nu}
\iota_*(c_\mu+c_\nu)=\iota(\eta^*)=0\in H_1(\mathbb T^3),
\end{equation}
where $\eta^*\in H_1(M)$ is the Poincare dual of the cohomology class of $\eta$.
We claim that $\iota_*(c_\mu)\ne0\ne\iota_*(c_\nu)$, which implies the assertion of the Proposition
about the existence of an asymptotic direction.

Indeed, suppose on the contrary that $\iota_*(c_\mu)=0$. Then for any $1$-cycle $c$ on $M$
that is null-homologous in $\mathbb T^3$ we must have $c\frown c_\mu=0$.
Let $c=\xi_2-\xi_3$, where $\xi_{2,3}$ are oriented transversals of $\mathcal F$ introduced in the proof
of Proposition~\ref{123} (consult Fig.~\ref{strips}, where initial portions of $\xi_{2,3}$ are shown).
The cycle $c$ is homologous to zero in $\mathbb T^3$, so we must have
$\int_{\xi_2}\mu=\int_{\xi_3}\mu$. Similarly, we must have
$\int_{\xi_2}\mu=\int_{\xi_4}\mu=\int_{\xi_5}\mu$, where
$$\begin{aligned}\xi_4&=p\bigl([(0,0,0),(0,0,1)]\cup[(0,0,1),(1,0,1)]\cup[(1,0,1)\cup(1,1,1)]\bigr),\\
\xi_5&=p\bigl([(0,0,0),(0,0,1)]\cup[(0,0,1),(0,1,1)]\cup[(0,1,1)\cup(1,1,1)]\bigr).
\end{aligned}$$
We have
\begin{equation}\begin{aligned}\label{urelations}
\int_{\xi_2}\mu&=u_2+u_3+u_4+u_5+u_6,&
\int_{\xi_3}\mu&=u_2+u_5+u_7+u_8+u_9,\\
\int_{\xi_4}\mu&=u_2+2u_3+u_4+u_6+u_7,&
\int_{\xi_5}\mu&=u_2+2u_3+u_6+2u_7.
\end{aligned}\end{equation}
So, $\overrightarrow u$ must satisfy the relations:
\begin{equation}
u_3+u_6=u_8+u_9,\quad u_5=u_3+u_7,\quad u_4=u_7.
\end{equation}
It must also satisfy~\eqref{match} (with $x_i$ replaced by $u_i$). The subspace in $\mathbb R^9$
defined by all these equations is invariant under $R(k)^{\pm1}$. Therefore,
they must hold true also for $\overrightarrow {u_\infty}$, but
the first relation in~\eqref{urelations} does not. Contradiction.

It follows from~\eqref{mu+nu} that $\iota_*(c_\mu)=-\iota_*(c_\nu)$, which implies that the asymptotic
direction of trajectories for $\mathcal E_\mu$ will be opposite to the one for $\mathcal E_\nu$.
\end{proof}

The hypothesis on the sequence $(k_i)$ in Proposition~\ref{123} is much weaker than in Proposition~\ref{prop-two-end}.
One can show that it can be weakened in Proposition~\ref{prop-two-end}, too, by deducing it from Proposition~\ref{123},
but the argument will be less straightforward.

We expect that all thin type band copmlexes with three bands give rise,
through the construction of~\cite{dyn08}, to
a chaotic dynamics in Novikov's problem with almost all trajectories having
an asymptotic direction, but don't see a rigorous proof of that so far.

\end{document}